\documentclass[12pt,leqno]{article}  
\usepackage{amsmath,amssymb,amsthm,bbm}
\usepackage{array,mathdots}
\usepackage{titlesec} 
\usepackage[all,cmtip]{xy}
\usepackage{blkarray,arydshln} 

\titlespacing{\section}{0cm}{3.5pc}{1.5pc}

\makeatletter
\def\@citex[#1]#2{\if@filesw\immediate\write\@auxout{\string\citation{#2}}\fi
  \def\@citea{}\@cite{\@for\@citeb:=#2\do
    {\@citea\def\@citea{\@citesep}\@ifundefined
       {b@\@citeb}{{\bf ?}\@warning
       {Citation `\@citeb' on page \thepage \space undefined}}%
{\csname b@\@citeb\endcsname}}}{#1}}
\def\@citesep{; }
\makeatother

\newtheoremstyle{Kang}{}{}{\itshape}{}{\bf}{}{.5em}{}
\theoremstyle{Kang}
\newtheorem{theorem}{Theorem}[section]

\newtheorem{lemma}[theorem]{Lemma}

\newtheoremstyle{Kremark}{}{}{}{}{\bf}{}{.5em}{}
\theoremstyle{Kremark}
\newtheorem*{remark}{Remark.}
\newtheorem{defn}[theorem]{Definition}

\newtheorem{other}{}

\allowdisplaybreaks[1]  

\def\fn#1{\operatorname{#1}} 
\def\bm#1{\mathbbm{#1}}
\def\c#1{\mathcal{#1}}

\title{A Note on Plans's Paper of Noether's Problem}

\author{Ming-chang Kang \\[3mm]
Department of Mathematics \\
National Taiwan University \\
Taipei, Taiwan \\
E-mail: kang@math.ntu.edu.tw}
\date{}

\begin{document}

\maketitle

\footnote{\textit{\!\!\! $2010$ Mathematics Subject
Classification}. 12F10, 13A50, 14E08.}
\footnote{\textit{\!\!\! Keywords and phrases}. Noether's problem, rationality problem, unramified primes.}

\begin{abstract}
{\noindent\bf Abstract.} Let $p$ be a prime number and $\zeta_p$ be a primitive $p$-th root of unity in $\bm{C}$.
Let $k$ be a field and $k(x_0,\ldots,x_{p-1})$ be the rational function field of $p$ variables over $k$.
Suppose that $G=\langle\sigma\rangle \simeq C_p$ acts on $k(x_0,\ldots,x_{p-1})$ by $k$-automorphisms defined as $\sigma:x_0\mapsto x_1\mapsto\cdots\mapsto x_{p-1}\mapsto x_0$.
Denote by $P$ the set of all prime numbers and define $P_0=\{p\in P:\bm{Q}(\zeta_{p-1})$ is of class number one$\}$.
Theorem. If $k$ is an algebraic number field and $p\in P\backslash (P_0\cup P_k)$,
then $k(x_0,\ldots,x_{p-1})^G$ is not stably rational over $k$ where $P_k=\{p\in P: p$ is ramified in $k\}$.
\end{abstract}

\section{Introduction}

Let $p$ be a prime number, $K$ be a field, $K(x_0,\ldots,x_{p-1})$ be the rational function field of $p$ variables over $K$.
Let $G=\langle\sigma\rangle\simeq C_p$ be the cyclic group of order $p$.
Suppose that $G$ acts on $K(x_0,\ldots,x_{p-1})$ by $K$-automorphisms defined as
\[
\sigma: x_0\mapsto x_1\mapsto \cdots\mapsto x_{p-1}\mapsto x_0.
\]
Define the fixed field $K(x_0,\ldots,x_{p-1})^G=\{f\in K(x_0,\ldots,x_{p-1}):\sigma\cdot f=f\}$.
The Noether's problem studied in \cite{Pl} is to investigate that, for which prime number $p$,
the fixed field $\bm{Q}(x_0,\ldots,x_{p-1})^G$ is rational (= purely transcendental) over $\bm{Q}$.

Denote by $\zeta_n$ a primitive $n$-th root of unity in $\bm{C}$ where $n$ is any positive integer.
Let $P$ be the set of all prime numbers.
Define $P_0=\{p\in P:\bm{Q}(\zeta_{p-1})$ is of class number one$\}$.
It is known that $P_0=\{p\in P:p\le 43\}\cup \{61,67,71\}$ by \cite{MM}.
Lenstra shows that the set of prime numbers $p$ such that $\bm{Q}(x_0,\ldots,x_{p-1})^G$ is rational is of Dirichlet density zero \cite[Corollary 7.6]{Le}.
The main result of Plans's paper \cite{Pl} is the following.

\begin{theorem}[{Plans \cite[Theorem 1.1]{Pl}}] \label{t1.1}
Let $p$ be a prime number.
Then $\bm{Q}(x_0,\ldots,x_{p-1})^G$ is rational over $\bm{Q}$ if and only if $p\in P_0$.
\end{theorem}

The purpose of this note is to study the situation when the base field is an algebraic number field.

Let $k$ be an algebraic number field and $A$ be the ring of algebraic integers in $k$,
A rational prime $p$ is unramified in $k$ if the ideal $pA$ is a finite product $\c{P}_1\cdots\c{P}_d$ where $\c{P}_1,\ldots,\c{P}_d$ are distinct prime ideals in $A$.
A prime number $p$ is called ramified in $k$ if it is not unramified.
Define $P_k=\{p\in P:p$ is ramified in $k\}$.
If $d_k$ denotes the discriminant of $k$,
then $P_k=\{p\in P:p\mid d_k\}$; thus $P_k$ is a finite set.

\medskip
The main result of this note is the following.

\begin{theorem} \label{t1.2}
Let $k$ be an algebraic number field.
If $p$ is a prime number and $p\in P\backslash (P_0\cup P_k)$, then $k(x_0,\ldots,x_{p-1})^G$ is not rational (resp.\ not stably rational) over $k$.
\end{theorem}

We remark that, if $p\in P_0$, then $k(x_0,\ldots,x_{p-1})^G$ is rational because $k(x_0,\ldots,x_{p-1})^G$ $=k(y_0,\ldots,y_{p-1})$ whenever $\bm{Q}(x_0,\allowbreak \ldots, x_{p-1})^G=\bm{Q}(y_0,\allowbreak \ldots, y_{p-1})$.
Also note that, $k(x_0,$ $\ldots,x_{p-1})^G$ is rational if and only if $k(x_0,\ldots,x_{p-1})^G$ is stably rational by \cite[Proposition 5.6]{Le}.

Theorem \ref{t1.2} can be generalized furthermore.

\begin{theorem} \label{t1.3}
Let $K$ be a field such that $\fn{char}K=0$ and $K$ is finitely generated over $\bm{Q}$. Then there is a finite subset $P'$ of $P$ satisfying the property that, for all prime numbers $p\in P\backslash (P_0 \cup P')$, $K(x_0,\ldots,x_{p-1})^G$ is not rational (resp.\ not stably rational) over $K$.
\end{theorem}

The proofs of Theorem \ref{t1.2} and Theorem \ref{t1.3} will be given in Section 2.

\section{Proof of Theorem \ref{t1.2}}

Let $\pi$ be a finite group.
We recall the definition of $\pi$-lattices.

\begin{defn} \label{d2.1}
Let $\pi$ be a finite group.
A finitely generated $\bm{Z}[\pi]$-module $M$ is called a $\pi$-lattice if $M$ is a free abelian group when it is regarded as an abelian group.

If $M$ is a $\pi$-lattice and $L$ is a field with $\pi$-action,
we will associate a rational function field over $L$ with $\pi$-action as follows.
Suppose that $M=\bigoplus_{1\le i\le m} \bm{Z}\cdot u_i$.
Define $L(M)=L(x_1,\ldots,x_m)$, a rational function field of $m$ variables over $L$.
For any $\sigma\in \pi$, if $\sigma\cdot u_i=\sum_{1\le j\le m} a_{ij}u_j$ in $M$ (where $a_{ij}\in\bm{Z}$),
we define $\sigma\cdot x_i=\prod_{1\le j\le m} x_j^{a_{ij}}$ in $L(M)$ and, for any $\alpha\in L$, define $\sigma \cdot \alpha$ by the prescribed $\pi$-action on $L$. \end{defn}

\begin{defn} \label{d2.2}
Let $\pi$ be a finite group and $M$ be a $\pi$-lattice.
$M$ is called a permutation lattice if $M$ has a $\bm{Z}$-basis permuted by $\pi$.
A $\pi$-lattice $M$ is called an invertible lattice if it is a direct summand of some permutation lattice.
A $\pi$-lattice $M$ is called a flabby lattice if $H^{-1}(\pi',M)=0$ for all subgroup $\pi'$ of $\pi$;
it is called a coflabby lattice if $H^1(\pi',M)=0$ for all subgroups $\pi'$ of $\pi$.
For the basic properties of $\pi$-lattices, see \cite{CTS,Sw}.
\end{defn}

\begin{defn} \label{d2.3}
Let $\pi$ be a finite group.
Two $\pi$-lattices $M_1$ and $M_2$ are called similar, denoted by $M_1\sim M_2$,
if $M_1\oplus Q_1\simeq M_2\oplus Q_2$ for some permutation lattices $Q_1$ and $Q_2$.
The flabby class monoid $F_\pi$ consists of all the similarity classes of flabby $\pi$-lattices under the addition described below.
Explicitly, if $M$ is a flabby $\pi$-lattice,
then $[M]\in F_\pi$ denotes the similarity class containing $M$;
the addition in $F_\pi$ is defined as: $[M_1]+[M_2]=[M_1\oplus M_2]$.
Note that $[M]=0$ in $F_\pi$,
i.e.\ $[M]$ is the zero element in $F_\pi$, if and only if $M\oplus Q$ is isomorphic to a permutation lattice where $Q$ is some permutation lattice.
See \cite{Sw} for details.
\end{defn}

\begin{defn} \label{d2.4}
Let $\pi$ be a finite group, $M$ be a $\pi$-lattice.
The $M$ have a flabby resolution,
i.e.\ there is an exact sequence of $\pi$-lattices:
$0\to M\to Q\to E\to 0$ where $Q$ is a permutation lattice and $E$ is a flabby lattice \cite[Lemma 1.1; CTS; Sw]{EM2}.

Although the above flabby resolution is not unique, the class $[E]\in F_\pi$ is uniquely determined by $M$.
Thus we define the flabby class of $M$, denoted as $[M]^{fl}$, by $[M]^{fl}=[E]\in F_\pi$ (see \cite{Sw}).
Sometimes we say that $[M]^{fl}$ is permutation or invertible if the class $[E]$ contains a permutation lattice or an invertible lattice.
\end{defn}

\begin{theorem} \label{t2.5}
Let $L/K$ be a finite Galois extension with $\pi=\fn{Gal}(L/K)$.
Let $M$ be a $\pi$-lattice.
\begin{enumerate}
\item[$(1)$] {\rm (\cite[Theorem 1.6; Vo; Le, Theorem 1.7; CTS]{EM1})}
The fixed field $L(M)^\pi$ is stably rational over $K$ if and only if $[M]^{fl}=0$ in $F_\pi$.
\item[$(2)$] {\rm (\cite[Theorem 3.14]{Sa})}
Assume that $K$ is an infinite field.
Then the fixed field $L(M)^\pi$ is retract rational over $K$ if and only if $[M]^{fl}$ is invertible.
\end{enumerate}
\end{theorem}

\begin{remark}
Note that, ``rational" $\Rightarrow$ ``stably rational" $\Rightarrow$ ``retract rational"$\Rightarrow$ ``unirational".
For the definition of retract rationality, see \cite{Sa}.
The fixed field $L(M)^\pi$ is the function field of the algebraic torus $T$ defined over $K$,
split by $L$ and with character module $M$ (see \cite{Vo,Sw}).
\end{remark}

\begin{lemma} \label{l2.6}
Let $k$ be an algebraic number field, $P_k$ be the finite subset of $P$ defined in Section 1.
If $p\in P\backslash P_k$, then $[k(\zeta_p):k]=p-1$ and $\fn{Gal}(k(\zeta_p)/k)\simeq C_{p-1}$.
\end{lemma}

\begin{proof}
Let $f(X)=\Phi_p(X)=X^{p-1}+\cdots+X+1\in k[X]$.
We will show that $f(X)$ is irreducible in $k[X]$.

Let $A$ be the ring of algebraic integers in $k$ and $pA=\c{P}_1\cdots \c{P}_d$ where $\c{P}_1,\ldots,\c{P}_d$ are distinct prime ideals (because $p$ is unramified in $k$).

Write $\c{P}=\c{P}_1$ and consider the localization $A_{\c{P}}$.
Note that $A_{\c{P}}$ is a DVR whose maximal ideal is generated by some prime element $v\in\c{P}A_{\c{P}}$.
From $pA=\c{P}_1\cdots \c{P}_d$, we get $p=\varepsilon v$ where $\varepsilon$ is a unit in $A_{\c{P}}$.

Now we begin to prove that $f(X)$ is irreducible in $k[X]$.
Note that $f(X)$ is irreducible in $k[X]$ if and only if so is $f(X+1)$.
It is easy to see that $f(X+1)=X^{p-1}+\sum_{1\le i\le p-1} a_i X^{p-i-1}$ where $a_i\in \bm{Z}$,
$a_{p-1}=p$ and $p\mid a_i$ for $1\le i\le p-1$.
Regard $f(X+1)$ as a polynomial in $A_{\c{P}}[X]$ and apply Eisenstein's irreducibility criterion.
It follows that $f(X+1)$ is irreducible in $A_{\c{P}}[X]$.
In particular, it is irreducible in $k[X]$.
\end{proof}

We may generalize Lemma \ref{l2.6} as follows.

\begin{lemma} \label{l2.7}

Let $K$ be a field such that $\fn{char}K=0$ and $K$ is finitely generated over $\bm{Q}$. Then there is a finite subset $P'$ of $P$ satisfying the property that, for all prime numbers $p\in P\backslash P'$, we have $[K(\zeta_p):K]=p-1$ and $\fn{Gal}(K(\zeta_p)/K)\simeq C_{p-1}$.
\end{lemma}

\begin{proof}

Step 1.
As before, we will show that $\Phi_p(X)$ is irreducible in $K[X]$ where $\Phi_p(X)=X^{p-1}+\cdots+X+1\in K[X]$.

Let $k$ be the algebraic closure of $\bm{Q}$ in $K$. Then $k$ is an algebraic number field. Choose a transcendence basis $t_1, \ldots, t_m$ of $K$ over $k$. Thus $k(t_1, \ldots, t_m)$ is rational over $k$ and $K$ is a finite extension of $k(t_1, \ldots, t_m)$.

For simplicity, we consider the case $m=1$ (the general case can be proved similarly). From now on, we consider the field extensions $\bm{Q} \subset k \subset k(t) \subset K$.

\medskip
Let $A$ be the ring of algebraic integers in $k$.

Define $S$ to be the multiplicatively closed subset of the polynomial ring $A[t]$ consisting of all the monic polynomials. Define $B= S^{-1}A[t]$, the localization of $A[t]$ by $S$. From Seidenberg's theory, we find that $B$ is of Krull dimension one (see \cite[page 26, Theorem 39]{Ka}). Thus $B$ is a Dedekind domain.

Define $C$ to be the integral closure of $B$ in $K$. Then $C$ is also a Dedekind domain. Note that there are only finitely many prime ideals $\mathcal{Q}_1, \ldots, \mathcal{Q}_t$ in $B$ such that $\mathcal{Q}_j$ are ramified in $C$ (see the last paragraph of \cite[page 306]{ZS}).

\bigskip
Step 2.
Let $P_k$ be the same subset of $P$ defined in Lemma \ref{l2.6} for the algebraic number field $k$. For any $p \in P \setminus P_k$, $p$ is unramified in $A$. Thus $p$ is also unramified in $B$. Now we enlarge $P_k$ to a finite subset $P'$ of $P$ by adding all the prime numbers $p$ such that $p\bm{Z}=\mathcal{Q}_j \cap \bm{Z}$ where $\mathcal{Q}_j$ belongs to those prime ideals ramified in $C$ (see Step 1).

It is not difficult to see that, if $p \in P \setminus P'$, then $p$ is unramified in $C$. For such prime number $p$, the polynomial $\Phi_p(X)$ is irreducible in $K[X]$ because the proof is the same as in Lemma \ref{l2.6}.
\end{proof}

\begin{proof}[Proof of Theorem \ref{t1.2}] ~ \par

Step 1.
Let $p$ be a prime number, $K$ be a field satisfying that $\fn{char}K=0$ and $[K(\zeta_p):K]=p-1$.
Consider the action of $G=\langle \sigma\rangle\simeq C_p$ on the rational function field $K(x_0,x_1,\ldots,x_{p-1})$ by
\[
\sigma:x_0\mapsto x_1\mapsto \cdots \mapsto x_{p-1}\mapsto x_0.
\]
Define $\pi=\fn{Gal}(K(\zeta_p)/K)=\langle \tau\rangle \simeq C_{p-1}$ where $\tau\cdot \zeta_p=\zeta_p^t$,
$(\bm{Z}/p\bm{Z})^{\times}=\langle \bar{t}\rangle$ with $t^{p-1}=1+ps$ (it is required that $t,s\in\bm{N})$.

Define $y_0,y_1,\ldots,y_{p-1}\in K(\zeta_p)(x_0,\ldots,x_{p-1})$ by
\[
y_i=\sum_{0\le j\le p-1} \zeta_p^{-ij} x_j.
\]
Extend the actions of $\sigma$ and $\tau$ to $K(\zeta_p)(x_0,\ldots,x_{p-1})$ by requiring that $\sigma\cdot \zeta_p=\zeta_p$,
$\tau\cdot x_i=x_i$ for $0\le i\le p-1$.
It follows that
\[
\sigma\cdot y_i=\zeta_p^i y_i,\quad \tau\cdot y_i=y_{it}
\]
for $0\le i\le p-1$ where the indices of $y_i$ are taken modulo $p$.
Thus we have
\begin{align*}
K(x_0,\ldots,x_{p-1})^{\langle \sigma\rangle} &= \{K(\zeta_p)(x_0,\ldots,x_{p-1})^{\langle\tau\rangle}\}^{\langle\sigma\rangle} \\
&=K(\zeta_p)(x_0,\ldots,x_{p-1})^{\langle\sigma,\tau\rangle} \\
&=K(\zeta_p)(y_0,\ldots,y_{p-1})^{\langle\sigma,\tau\rangle} \\
&=\{K(\zeta_p)(y_0,\ldots,y_{p-1})^{\langle\sigma\rangle}\}^{\langle\tau\rangle} \\
&=K(\zeta_p)(z_1,\ldots,z_{p-1})(y_0)^{\langle\tau\rangle}
\end{align*}
where $z_1=y_t/y_1^t$, $z_2=y_{t^2}/y_t^t$, $\ldots$, $z_{p-2}=y_{t^{p-2}}/(y_{t^{p-3}})^t$, $z_{p-1}=y_1^p$ (with the indices of $y_i$ being taken modulo $p$).
Note that
\begin{align*}
\tau:{}& \zeta_p\mapsto \zeta_p^t,~z_1\mapsto z_2\mapsto\cdots\mapsto z_{p-2}\mapsto (z_1^{t^{p-2}} z_2^{t^{p-3}} \cdots z_{p-3}^{t^2} z_{p-2}^t z_{p-1}^s)^{-1}, \\
& z_{p-1}\mapsto z_1^p z_{p-1}^t,~ y_0\mapsto y_0
\end{align*}
and recall that $t^{p-1}=1+ps$ defined before.

It follows that $K(x_0,\ldots,x_{p-1})^G=K(\zeta_p)(M)^\pi (y_0)$ where $M$ is the $\pi$-lattice defined by $M=\bigoplus_{1\le i\le p-1} \bm{Z}\cdot w_i$
and $\tau:w_1\mapsto w_2\mapsto\cdots\mapsto w_{p-2}\mapsto -(t^{p-2}\cdot w_1+t^{p-3}\cdot w_2+\cdots+t^2\cdot w_{p-3}+t\cdot w_{p-2}+s\cdot w_{p-1})$,
$w_{p-1}\mapsto p\cdot w_1+t\cdot w_{p-1}$.

Apply Theorem \ref{t2.5}.
We conclude that $K(x_0,\ldots,x_{p-1})^G$ is stably rational over $K$ if and only if $[M]^{fl}=0$ in $F_\pi$.

\medskip
Step 2.
Now consider the case where $k$ is an algebraic number field and $p\in P\backslash (P_0\cup P_k)$.
We will show that $k(x_0,\ldots,x_{p-1})^G$ is not stably rational over $k$.

Since $p\notin P_k$, by Lemma \ref{l2.6}, we find that $\fn{Gal}(k(\zeta_p)/k)\simeq C_{p-1}$, which is isomorphic to $\fn{Gal} (\bm{Q}(\zeta_p)/\bm{Q})$.
Call this group $\pi$.

Apply Step 1.
We find that we arrive at the same $\pi$-lattice $M$ for $k(x_0,\ldots,x_{p-1})^G$ and $\bm{Q}(x_0,\ldots,x_{p-1})^G$.

From Theorem \ref{t1.1}, $\bm{Q}(x_0,\ldots,x_{p-1})^G$ is not rational.
Thus it is not stably rational over $\bm{Q}$ by \cite[Proposition 5.6]{Le}.
It follows that $[M]^{fl}\ne 0$ in $F_\pi$ by the conclusion of Step 1.

Apply Theorem \ref{t2.5} to $k(\zeta_p)(M)^\pi$ and use the above result that $[M]^{fl}\ne 0$.
We find that $k(x_0,\ldots,x_{p-1})^G=k(\zeta_p)(M)^\pi$ is not stably rational over $k$.
\end{proof}

\begin{remark}
Because $\pi$ is a cyclic group, any flabby $\pi$-lattice is invertible by Endo and Miyata \cite[Theorem 1.5]{EM2}.
Applying Theorem \ref{t2.5}, we find that $K(\zeta_p)(M)^\pi$ is retract rational over $K$ in Step 1.
More generally, it can be shown that $K(x_0,\ldots,x_{p-1})^G$ is retract rational over $K$ for any field $K$ and for any prime number $p$. The proof is omitted.
\end{remark}

\begin{proof}[Proof of Theorem \ref{t1.3}] ~ \par

The proof is almost the same as that of Theorem \ref{t1.2}, except that we apply Lemma \ref{l2.7} this time.
\end{proof}



\begin{thebibliography}{HKK}

\bibitem[CTS]{CTS}
J.-L. Colliot-Th\'el\`ene and J.-J. Sansuc,
\textit{La $R$-\'equivalence sur les tores},
Ann. Sci. \'Ecole Norm. Sup. (4) 10 (1977), 175--229.

\bibitem[EM1]{EM1}
S. Endo and T. Miyata, \textit{Invariants of finite abelian
groups}, Journal Math. Soc. Japan 25 (1973), 7--26.

\bibitem[EM2]{EM2}
S. Endo and T. Miyata, \textit{On a classification of the function
fields of algebraic tori}, Nagoya Math. J. 56 (1975), 85--104.

\bibitem[Ka]{Ka}
I. Kaplansky,
\textit {Commutative rings, revised edition}, The University of Chicago Press, Chicago, 1974.

\bibitem[Le]{Le}
H. W. Lenstra Jr.,
\textit{Rational functions invariant under a finite abelian group},
Invent. math. 25 (1974), 299--325.

\bibitem[MM]{MM}
J. M. Masley and H. L. Montgomery,
\textit{Cyclotomic fields with unique factorization},
J. Reine Angew. Math. 286/287 (1976), 248--256.

\bibitem[Pl]{Pl}
B. Plans, \textit{On Noether's rationality problem for cyclic groups over $\bm{Q}$}, arXiv: 1605.09228.

\bibitem[Sa]{Sa}
D. J. Saltman,
\textit{Retract rational fields and cyclic Galois extensions},
Israel J. Math. 47 (1984), 165--215.

\bibitem[Sw]{Sw}
R. G. Swan, \textit{Noether's problem in Galois theory},
in ``Emmy Noether in Bryn Mawr", edited by B. Srinivasan and J. Sally,
Springer-Verlag, Berlin, 1983, pp.~21--40.

\bibitem[Vo]{Vo}
V. E. Voskresenskii, \textit{Algebraic groups and their birational invariants},
Amer. Math. Soc. Transl. of Math. Monographs, vol. 179, Providence, 1998.

\bibitem[ZS]{ZS}
O. Zariski and P. Samuel, \textit{Commutative algebra, vol. 1}, Van Nostrand, 1958; reprinted by Springer-Verlag, 1975, Berlin.

\end{thebibliography}
\end{document}